\documentclass[10pt]{article}
\usepackage{amsmath,amscd}
\usepackage{amssymb,latexsym,amsthm}
\usepackage{color,palatino}
\usepackage[spanish,english]{babel}
\usepackage{multirow}
\usepackage[pdftex]{hyperref}
\usepackage[latin1]{inputenc}
\usepackage{lscape}
\usepackage{enumerate}
\usepackage{fancyhdr}
\usepackage{pb-diagram}
\usepackage{amsfonts}

\numberwithin{equation}{section}
\newtheorem{theorem}{Theorem}[section]

\newtheorem{definition}[theorem]{Definition}
\newtheorem{proposition}[theorem]{Proposition}
\newtheorem{lemma}[theorem]{Lemma}

\newtheorem{corollary}[theorem]{Corollary}

\theoremstyle{definition}

\newtheorem{remark}[theorem]{Remark}

\newcommand\h{\mathfrak h}

\newcommand\U{\mathcal U}
\newcommand\B{\mathcal B}
\newcommand\J{\mathcal J}
\newcommand\M{\mathcal M}
\newcommand\PP{\mathcal P}
\newcommand\LL{\mathcal L}

\newcommand\vu{\textbf{\emph{u}}}

\newcommand\K{\Bbbk}

\makeatletter
\def\@roman#1{\romannumeral #1}
\makeatother

\title{\textbf{Stable rank of down-up algebras}}
\author{Claudia Gallego and Andrea Solotar\thanks{This work has been supported by the projects UBACYT 20020130 100533BA, PIP-CONICET
11220150100483CO, and PICT 2015-0366. The first named author is a CONICET postdoctoral fellow. The second named author is a research member of
\-CO\-NI\-CET (Argentina) and a Senior Associate of ICTP Associate Scheme.}}
\date{}
\begin{document}
\makeatletter
\def\@roman#1{\romannumeral #1}
\makeatother
\maketitle
\begin{abstract}
\noindent
We investigate the behavior of finitely generated projective modules over a down-up algebra. Specifically, we show that every noetherian down-up algebra $A(\alpha,\beta,\gamma)$ has a non-free, stably free right ideal. Further, we compute the stable rank of these algebras using Stafford's Stable Range Theorem and Kmax dimension.
\bigskip

\noindent
\textit{Keywords:} Down-up algebras, stably free modules, projective modules, stable rank, Krull dimension, Kmax dimension.

\bigskip

\noindent 2010 \textit{Mathematics Subject Classification.} Primary: 16D40, 19A13, 19B10.  Secondary: 16P60.
\end{abstract}
\section{Introduction}
\noindent
The study of finitely projective modules over an arbitrary ring is a classical task in homological algebra. Investigating whether these modules are free, or
at least stably free,  has also great interest in  geometry, topology and $K$-theory. One of the most well known results in this context is the Quillen-Suslin theorem about Serre's problem for the commutative polynomial ring $\K[x_1,\ldots,x_n]$, where $\K$ is a field. In this particular situation, Quillen and Suslin proved independently that the finitely generated projective modules are free, see \cite{Lam} for a detailed and very clear exposition about this subject. However, for noncommutative rings of polynomial type it is easy to present examples where the Quillen-Suslin Theorem fails. For instance, if $T$ is a division ring and $S := T[x,y]$, there is an $S$-module $M$ such that $M\bigoplus S\cong S^2$, but $M$ is not free, see \cite{Stafford5}. Moreover, Stafford developed conditions in
\cite[Theorem 1.2]{Stafford5} under which the skew polynomial ring $S=R[x;\sigma,\delta]$, with $R$ a noetherian domain, $\sigma$ an automorphism of $R$ and $\delta$ a $\sigma$-derivation, has a non-trivial stably free right ideal. These ideas have been used in \cite{Antoniou} in order to obtain non-trivial stably free modules over the enveloping algebras of the RIT (relativistic internal time) Lie algebras. Using similar methods, Iyudu and Wisbauer  gave a sufficient condition in \cite{Iyudu} for the existence of projective non-free modules over the class of crossed products of noetherian domains with universal enveloping algebras of Lie algebras. In the current paper, we will show that there exist non-free projective modules over down-up algebras too. This fact will allow us to obtain bounds of the stable rank of these algebras. \\
\\
Down-up algebras have been introduced by Benkart and Roby in \cite{Benkart1} motivated by the study of posets. Given a field $\K$ and constants $\alpha$, $\beta$, $\gamma$ in $\K$, the \emph{down-up algebra} $A=A(\alpha,\beta,\gamma)$ is the associative algebra generated over $\K$ by $U$ and $D$, subject to the defining relations:
\begin{align*}
DU^2=&\alpha UDU+\beta U^2D+\gamma U\\
D^2U=&\alpha DUD+\beta UD^2+\gamma D.
\end{align*}
As known examples of down-up algebras, we can mention $A(2,-1,0)$ that turns out to be isomorphic to the enveloping algebra of the Heisenberg Lie algebra of dimension 3; for the case  where $\gamma\neq 0$, the algebra $A(2,-1,\gamma)$ is isomorphic to the enveloping algebra of $sl_2(\K)$. For another interesting example, consider the quantized enveloping algebra $U_q(sl_3(\K))$ with generators $E_i$, $F_i$, $K^{\pm 1}$, $i=1$, $2$ and a non-zero scalar $q$ in $\K$; the subalgebra of $U_q(sl_3(\K))$ generated by $E_1$, $E_2$ is the down-up algebra $A([2]_{q},-1,0)$, where $[n]_{q}=\frac{q^{n}-q^{-n}}{q-q^{-1}}$. For $\gamma\neq 0$, the down-up algebra $A(0,1,\gamma)$ is isomorphic to the  enveloping algebra of the Lie superalgebra osp(1,2).\\
\\
Kirkman, Musson and Passman proved in \cite{Kirkman}  that $A(\alpha,\beta,\gamma)$ is a noetherian algebra if and only if,  the parameter $\beta$ is non-zero; the latter is equivalent to saying that $A(\alpha,\beta,\gamma)$ is a domain. Furthermore, for down-up algebras, the Krull, Gelfand-Kirillov, and global dimensions have already been computed, see \cite{Bavula4}, \cite{Benkart1} and \cite{Kirkman}. Additionally, their representation theory, Hochschild homology and cohomology, as well as several homological and ring theoretical properties have also been studied  (e.g., \cite{carvalho}, \cite{carvalho2}, \cite{solotar}, \cite{Jordan}, \cite{zhao}).\\
\\
In \cite{Benkart1} the task of investigating indecomposable and projective modules for down-up algebras was proposed. We give a partial answer to this subject proving in Theorem \ref{T0} that all finitely generated projective modules over a noetherian down-up algebra are stably free. Moreover, we show that the class of noetherian down-up algebras does not satisfy a noncommutative version of Quillen-Suslin Theorem in the sense that there exist non-trivial stably free modules over these algebras, see Corollary \ref{C1} and Proposition \ref{P1}. In view of the above, we obtain a lower bound of the stable rank of a down-up algebra and, using the Stafford's Stable Range Theorem, we achieve in Theorem \ref{P2} upper bounds of this value. Finally, under certain conditions, the exact value of stable rank is obtained in Theorem \ref{P3}.\\
\\
The article is organized as follows: in Section \ref{s2} we prove that every finitely generated projective module over a noetherian down-up algebra is stably free.\\
Section \ref{s3} is devoted to showing that the algebra $A=A(\alpha,\beta,\gamma)$, with $\beta\neq 0$, always has a non-trivial stably free right ideal. For this task, we split the problem in two cases: $\gamma\neq 0$ and $\gamma=0$ and we use some techniques from \cite{Stafford5} to achieve our goal.\\
In Section \ref{s4}, bounds of the stable rank of a down-up algebra are established. Under some conditions over the roots of the polynomial $t^2-\alpha t-\beta$, such bounds are improved. The main tool at this point is the Kmax dimension of an arbitrary ring.

\section{Stability of projective modules}\label{s2}
\noindent
A ring $S$ is called a \emph{PSF} ring if every finitely generated projective $S$-module is stably free. In this section we will show that, for $\beta\neq 0$, the algebra $A=A(\alpha,\beta,\gamma)$ is a \emph{PSF} ring. It is important to note that, as $\beta$ is non-zero, $A$ is a right (left) noetherian ring \cite[Corollary 2.2]{Kirkman} and, therefore, the rank of free $A$-modules and the rank of  stably free $A$-modules are well defined.
\begin{theorem}\label{T0}
Let $A=A(\alpha,\beta,\gamma)$ be a down-up algebra. If $\beta\neq 0$, then $A$ is a \emph{PSF} ring.
\end{theorem}
\begin{proof}
In \cite[Section 3.1]{Kirkman} it is proved that the collection $\{V_n\}_{n\geq 0}$ given by $V_0:=\Bbbk$, $V_1:=\Bbbk+\K u+\K d$ and $V_n:= (V_1)^n$, for $n\geq 2$, is a filtration of $A$, and that  $Gr(A)$, the associated graded ring, is isomorphic to the down-up algebra $A(\alpha,\beta,0)$. Hence, $Gr(A)$ is a right (left) noetherian ring. It is also known that if $A$ is noetherian, then gldim($A$) $=3$ \cite[Theorem 4.1]{Kirkman}; thus  $Gr(A)$ is a right regular ring. Since $Gr(A)$ is a free $V_0$-module, it follows from \cite[Theorem 12.3.2]{McConnell} that $A$ is a \emph{PSF} ring.
\end{proof}
\begin{remark}
Given $R=\bigoplus_{i\geq 0}R_i$ a graded ring, we know that if $P$ is a finitely generated graded projective $R$-module, then $P$ is extended from $R_0$; more precisely, there is a graded $R$-module isomorphism $R\otimes_{R_0} P_0\cong P$, where $P_0$ is a graded projective $R_0$-module, see \cite[Theorem II. 4.6]{Lam}. Thus, if $A=A(\alpha,\beta,\gamma)$ is a noetherian down-up algebra with $\gamma=0$, every finitely generated graded projective $A$-module $P$ is extended  from $\K$. Hence, $P$ turns out to be a free $A$-module.
\end{remark}
\section{Non-trivial stably free ideals}\label{s3}
\noindent
It is well known that there exist stably free modules which are non-free over $\U(sl_2(\K))$ and $\U(\h)$, where $\h$ denotes the Heisenberg Lie algebra of dimension 3 \cite{Stafford5}. These algebras are examples of down-up algebras, so this raises the question whether every down-up algebra has a non-trivial stably free module or not. The goal of this section is to exhibit examples of such modules. To achieve such objective, we will distinguish two cases: $\gamma\neq 0$  and $\gamma=0$. In the following, we assume that $\beta\neq 0$, and that $\K$ is a field of characteristic zero that contains both roots of the polynomial $t^2-\alpha t-\beta$.
\vspace{-2mm}
\subsection{Case $\gamma\neq 0$}
\noindent
For this case, we will proceed as in \cite{Antoniou}, \cite{Iyudu} and \cite{Stafford5}: first, we consider a subalgebra $\widetilde{A}$ of $A$ for which there exists a right (left) non-trivial stably free ideal $K$. Afterwards, we extend such ideal $K$ to the whole algebra $A$ using results from \cite{Stafford5}.\\
Let $\lambda$ and $\mu$ be the roots of $t^2-\alpha t-\beta$, so that $\alpha=\lambda+\mu$ and $\beta=-\lambda\mu$. Since $\beta$ is non-zero, it follows that $\lambda$ and $\mu$ are both non-zero. For $\gamma\neq 0$, there is an isomorphism $A(\alpha,\beta,\gamma)\cong A(\alpha,\beta,1)$, see \cite[Lemma 4.1 (ii)]{carvalho}, so we assume $\gamma=1$ without loss of generality. Under these conditions, the multiplication rules in $A$ are given by:
\begin{align*}
[D,[D,U]_{\lambda}]_{\mu}=D \quad\quad\quad\quad [[D,U]_{\lambda},U]_{\mu}=U,
\end{align*}
where $[a,b]_{\eta}$ denotes the expression $ab-\eta ba$. Let $\omega:=[d,u]_{\lambda}=du-\lambda ud$ and consider the algebra $\widetilde{A}:=\K[u][\omega;\sigma,\delta]$ with $\sigma$ an automorphism of $\K[u]$ such that $\sigma(u):=\mu^{-1}u$, and $\delta$ the $\sigma$-derivation defined by $\delta(u):=-\mu^{-1}u$. Let us see that $\widetilde{A}$ is a subalgebra of $A$: indeed, let $\phi:\widetilde{A}\to A$ be determined by $u\mapsto U$ and $\omega\mapsto [D,U]_{\lambda}$; then:
\begin{align*}
U[D,U]_{\lambda}&=U(DU-\lambda UD)=UDU-\lambda U^2D\\
&=\mu^{-1}(\mu UDU-\lambda \mu U^2D-U)=\mu^{-1}(DU^{2}-\lambda UDU-U)\\
&=\mu^{-1}[D,U]_{\lambda}U-\mu^{-1}U,
\end{align*}
and therefore, $\phi$ turns out to be an algebra homomorphism. The set $\B_1:=\{u^{i}\omega^j\mid i,j\in \mathbb{N}\}$ is a $\K$-basis for $\widetilde{A}$ and $\phi(\B_1)=\{U^j(DU-\lambda UD)^j\mid i,j\in \mathbb{N}\}$. Since $\mathcal{B}:=\{U^i(DU+aUD+b)^jD^k\mid i,j,k\in\mathbb{N}\}$ is a $\K$-basis of $A$ for any $a,b\in \K$ \cite[Lemma 2.2]{zhao}, then $\phi(\B_1)$ is linearly independent in $A$ and $\widetilde{A}$ is a subalgebra of $A$.\\
\\
We will strongly  use the following remarkable result from \cite{Stafford5}:
\begin{lemma}\cite[Corollary 1.6]{Stafford5}\label{L0}
Let $R$ be a noetherian domain and let $S=R[x;\sigma,\delta]$ be an Ore extension. Suppose that there exists a non-unit $r\in R$ such that $\sum_{i\geq 0}\delta^{i}(r)R=R$. Then $K=rS\cap xS$ is a non-trivial, stably free right ideal of $S$.
\end{lemma}
This latter lemma will allow us to carry out the first step to achieve our goal.
\begin{lemma}\label{L1}
The subalgebra $\widetilde{A}$ has a stably free right ideal that is non-free.
\end{lemma}
\begin{proof}
In view of the fact that $\K[u]$ is a domain and $\sigma$ is an automorphism, we have that $\widetilde{A}$ is also a domain. The element $r=1+u$ is non-invertible in $\K[u]$, with the property that
\begin{align*}
r+\delta(r)\mu=1+u+(-\mu^{-1}u)\mu=1+u-u=1.
\end{align*}
Thus, $\sum_{i\geq 0}\delta^{i}(r)\K[u]=\K[u]$ and Lemma \ref{L0} asserts that $\widetilde{A}$ has a right stably free ideal $K$ which is non-free. The ideal $K$ is defined by $\{f\in \widetilde{A}\mid rf\in \omega \widetilde{A}\}$ and  is isomorphic to $r\widetilde{A}\cap \omega\widetilde{A}$. Moreover, in the proof of Lemma \ref{L0} it is proved that $\widetilde{A}=r\widetilde{A}+\omega\widetilde{A}$; specifically, we have $1=r(1+\mu \omega)+\omega(-\mu\sigma(r))$. This equality allows us to obtain generators for the right ideal $K$: in effect, we claim that $a=(\omega+\mu^{-1})(1+u)-\mu^{-1}$ and $b=\omega^2+\mu^{-1}\omega$ are polynomials such that $K=a\widetilde{A}+b\widetilde{A}$. To show this, we first note that:
\begin{align*}
ra&=(u+1)(\omega+\mu^{-1})(u+1)-\mu^{-1}(u+1)=u\omega(u+1)+\omega(u+1)+\mu^{-1}(u+1)^{2}-\mu^{-1}(u+1)\\
&=\omega(u+1)(\mu^{-1}u+1)-\mu^{-1}((u+1)^2-u(u+1)-(u+1))=\omega(u+1)(\mu^{-1}u+1)\in \omega \widetilde{A},
\end{align*}
and,
\begin{align*}
rb&=\omega(\omega+\mu^{-1})+u\omega^2+\mu^{-1}u\omega=\omega(\omega+\mu^{-1})+\mu^{-1}\omega u\omega-\mu^{-1}u\omega+\mu^{-1}u\omega\\
&= \omega(\omega+\mu^{-1}u\omega+\mu^{-1})\in \omega \widetilde{A}.
\end{align*}
Suppose that $\B_1$ is ordered by the deglex order $\prec$ with $u\prec \omega$, and let $f$ be a non-zero element in $K$. We claim that if $lm(f)=u^{\delta_1}\omega^{\delta_2}$ is the leading monomial of $f$, then $\delta_1+\delta_2\geq 2$ and $\delta_2\geq 1$: indeed, it is clear that either $\delta_1$ or $\delta_2$ is non-zero. If $\delta_1+\delta_2=1$, we have that $\delta_1=0$ or $\delta_2=0$. In the first case, $f=c_1\omega+c_2u+c_3$ with $c_i\in \K$ for $i=1,2,3$ and $c_1$  not zero. Then,
\begin{align*}
rf=&(u+1)(c_1\omega+c_2u+c_3)= c_1u\omega+c_2u^2+c_3u+c_1\omega+c_2u+c_3\\
=& c_1(\mu^{-1}\omega u-\mu^{-1}u) +c_2u^2+(c_3+c_2)u+c_1\omega+c_3\\
=& \omega(c_1\mu^{-1}u+c_1)+c_2u^2+(c_2+c_3-c_1\mu^{-1})u+c_3\in \omega \widetilde{A}.
\end{align*}
Therefore, $c_1=c_2=c_3=0$ and $f=0$, which is a contradiction. A similar result is obtained if we assume $\delta_1=1$ and $\delta_2=0$; hence $\delta_1+\delta_2\geq 2$. Now, suppose $\delta_2=0$; in such case $f=c_{\delta}u^{\delta}+f_1$, where $lm(f_1)\prec lm(f)$, $\delta\geq 2$ and $c_\delta\neq 0$. So,
\begin{align*}
rf=&(u+1)(c_{\delta}u^{\delta}+f_1)=c_{\delta}u^{\delta+1}+uf_1+f\in \omega A;
\end{align*}
in order that $rf\in \omega A$ necessarily $c_{\delta}=0$, but this contradicts our choice of $f$. Consequently, $f=c_{\delta}u^{\delta_1}\omega^{\delta_2}+ f_1$ where $\delta_1+\delta_2\geq 2$, $\delta_2\geq 1$, $lm(f_1)\prec lm(f)$ and $c_{\delta}$ is a non-zero scalar. In these conditions, $lm(f)$ is divisible by $lm(a)=u\omega$ or $lm(b)=\omega^{2}$. Applying a right division algorithm, we get $f=aq_1+bq_2+h$, where $h$  is reduced with respect to $a$ and $b$. If $h\neq 0$, we have that $lm(h)$ is not divisible neither by $lm(a)$ nor $lm(b)$; i.e., if $lm(h)=u^{\epsilon_1}\omega^{\epsilon_2}$, then $\epsilon_2=0$ or $\epsilon_1+\epsilon_2\leq 1$. But $h=f-aq_1-bq_2\in K$ and we obtain a contradiction. Whence $h=0$, $f=aq_1+bq_2$ and $K=a\widetilde{A}+b\widetilde{A}$.
\end{proof}
\begin{remark}
Lemma \ref{L0} is a corollary of a more general result by Stafford \cite[Theorem 1.2]{Stafford5}: given a noetherian domain $R$ and $S=R[x;\sigma,\delta]$, with $\sigma$ an automorphism of $R$ and $\delta$ a $\sigma$-derivation, if there exists a non-unit $r$ in $R$ and some $s\in R$ such that $S=rS+(x+s)S$, then $S$ has a non-trivial stably free right ideal. This assertion also has a version for Laurent skew polynomial rings and it was used as a unified way for producing non-trivial stably free right ideals over Weyl algebras, rings of polynomials with coefficients in a division ring in at least two variables, group rings of  poly (infinite cyclic) groups and enveloping algebras of non-abelian finite dimensional Lie algebras. However, these modules do not always exist: for example, if $R$ is a division ring, any projective module over $S$ is free (see \cite[Proposition 11.5.3]{McConnell}). Moreover, given $S=R[x;\delta]$ with $R$ a commutative local ring with maximal ideal $Q$ and $\delta$ a non-zero derivation of $R$, it is proved in \cite[Corollary 4.6]{Stafford5} that every stably free right ideal of $S$ is free if and only if $\delta(Q)\subseteq Q$ and Kdim($R$) $=1$.
\end{remark}
The second step is extending this right ideal $K$ to $A$ in such a way that we obtain a non-trivial, stably free right ideal of $A$. To achieve this, we will use the following fact.
\begin{proposition}\cite[Proposition 2.3]{Stafford5}\label{T1}
Let $A$ and $B$ be domains such that $A\subset B$ and $B$ is faithfully flat as left $A$-module, satisfying the following property:
\begin{center}
If $a$ and $b$ are non-zero elements of $B$ such that $ab\in A$, then $a=a_1c$ \quad\quad\quad ($\clubsuit$)\\
and $b=c^{-1}b_1$ for some unit $c$ in $B$ and elements $a_1,b_1\in A$.
\end{center}
Under these conditions, if $P$ is a projective right ideal of $A$ that is not cyclic, then $PB\cong P\otimes_{A}B$ is a projective right ideal of $B$ that is also non-cyclic. Further, if $P$ is stably free then so is $PB$.
\end{proposition}
\begin{proposition}
The rings $\widetilde{A}$ and $A$ are domains such that $\widetilde{A}\subset A$ and they satisfy the hypotheses of Theorem \ref{T1}.
\vspace{-2mm}
\begin{proof}
Since $\beta$ is non-zero, both $A$ and $\widetilde{A}$ are domains. Inasmuch as  $\mathcal{B}=\{u^i(du-\lambda du)^jd^k\mid i,j,k\in\mathbb{N}\}$ is a $\K$-basis of $A$, it follows that $\B_2:=\{d^{k}\mid k\in \mathbb{N}\}$ is an $\widetilde{A}$-basis for $A$ as a left $\widetilde{A}$-module: indeed, it is obvious that $\B_2$ generates \,$_{\widetilde{A}}A$. Taking into account that $A$ is a domain, in order to prove linear independence, it is enough to show that if $\sum_{l=1}^ma_ld^l=0$, with $a_l\in\widetilde{A}$, $1\leq l\leq m$, then $a_l=0$ for each $l$. However, $a_l=\sum_{k=1}^{s_l}c_{k}^{(l)}u^{i_k^l}\omega^{j_k^l}$ for certain elements $c_{k}^{(l)}\in\K\setminus\{0\}$, thus
\[0=\sum_{l=1}^ma_ld^l=\sum_{l=1}^{m}\sum_{k=1}^{s_l}c_{k}^{(l)}u^{i_k^l}\omega^{j_k^l}d^l.\]
Since $u^{i_k^l}\omega^{j_k^l}d^l\in \B$, we get $\sum_{l=1}^{m}\sum_{k=1}^{s_l}c_{k}^{(l)}u^{i_k^l}\omega^{j_k^l}d^l=\sum_t d_tx^{\alpha_t}$, with $x^{\alpha}\in \B$ and $d_t:=\sum_{x^{\alpha_t}=u^{i_k^l}\omega^{j_k^l}d^l}c_{k}^{(l)}$. As a consequence, $d_{t}=0$ for all $t$. Note that, given $l$ and $t$, there exists just one $u^{i_k^l}\omega^{j_k^l}\in \B_1$ such that $x^{\alpha_{t}}=u^{i_k^l}\omega^{j_k^l}d^l$, whence the set $\{d_t\}$ coincides with $\{c_{k}^{(l)}\}$. Therefore, all $c_{k}^{(l)}=0$ and $a_l=0$ for all $l$. So, $A$ is $\widetilde{A}$-free and, in particular,  $A$ turns out to be a faithfully flat left $\widetilde{A}$-module. Finally, to prove that condition ($\clubsuit$) is satisfied, we define the following subsets of $A$: set $F_0:=\widetilde{A}$ and $F_n:=F_0U_n$ for $n\geq 1$, where $U_{n}:=\,_\K\langle d^k \mid k\leq n\rangle$. It is clear that $A=\bigcup_{n\in\mathbb{N}}F_n$ and $F_p\subseteq F_q$ for $p<q$. Using multiplication rules in $A$ we obtain  $F_{p}F_q\subseteq F_{p+q}$, and it follows that $\{F_n\}_{n\in \mathbb{N}}$ is a filtration of the algebra  $A$. Let $f,g\in A$ be non-zero elements such that $fg\in F_0$. Since $A=\bigcup_{n\in\mathbb{N}}F_n$, there exist $p$ and $q\in \mathbb{N}$ with the property that $f\in F_{p}\setminus F_{p-1}$ and $g\in F_q\setminus F_{q-1}$. In this way, $f=\sum_{\delta}f_{\delta}d_{\delta}\in F_p$ and $g=\sum_{\epsilon}g_{\epsilon}d_{\epsilon}\in F_q$, where $f_{\delta},g_{\epsilon}\in F_0$, $d_{\delta}\in U_p$ and $d_\epsilon\in U_q$. Hence, $f=a_1x^{\overline{\delta}_1}d^{\delta_1}+\cdots+a_rx^{\overline{\delta}_s}d^{\delta_s}$ with $x^{\overline{\delta}_i}\in \B_1$ and $\delta_i\leq p$ for each $i$; analogously, $g=b_1x^{\overline{\epsilon}_1}d^{\epsilon_1}+\cdots+b_tx^{\overline{\epsilon}_t}d^{\epsilon_t}$, with $x^{\overline{\epsilon}_j}\in \B_1$, and  $\epsilon_j \leq q$. We can suppose $d^{\delta_1}=d^{p}$ and $d^{\epsilon_1}=d^{q}$, so $fg=a_1b_1x^{\overline{\delta}_1+\overline{\epsilon}_1}d^{p+q}+f_0g_0$, with $f_0g_0$ a polynomial in $d$, both of degree less or equal to $p+q$. But $fg\in F_0$, so $p+q=0$ and $\delta_i=\epsilon_j=0$ for all $i$ and $j$, i.e., $f,g\in F_0$. This finishes the proof.
\end{proof}
\end{proposition}
\begin{corollary}\label{C1}
If $A=A(\alpha,\beta,\gamma)$ is a down-up algebra with $\gamma\neq 0$, then $A$ has a non-trivial, stably free right ideal.
\end{corollary}
\begin{proof}
By Lemma \ref{L1}, the algebra $\widetilde{A}$ has a stably free right ideal $K$ that is not free. Since  $\widetilde{A}$ and $A$ satisfy the hypotheses of Theorem \ref{T1}, the right ideal $KA\cong K\otimes_{\widetilde{A}}A$ is a non-trivial stably free right ideal of $A$. In the proof of Lemma \ref{L1}, we showed that $K=a\widetilde{A}+b\widetilde{A}$ with $a=(\omega+\mu^{-1})(1+u)-\mu^{-1}$ and $b=\omega^2+\mu^{-1}\omega$; thus $KA=(a\widetilde{A}+b\widetilde{A})A$. We shall prove that $KA=aA+bA$: given a non-zero polynomial $f\in aA+bA$, there exist $f_1$, $f_2\in A$ such that $f=af_1+bf_2$. Writing $f_1$ and $f_2$ in terms of $\B$, we have $f=\sum c_iax^{\delta_i}+\sum e_j bx^{\epsilon_j}$. Note that each term in the last expression can be written as $(\lambda_1ax^{\delta'}+\lambda_2bx^{\delta'})x^{\delta}$, where $x^{\delta'}\in \B_1$, $x^{\delta}\in \B_2$ and $\lambda_1, \lambda_2$ are not both zero.
Hence, $f$ can be expressed as a sum of elements in $(a\widetilde{A}+b\widetilde{A})A$ which implies that $f\in (a\widetilde{A}+b\widetilde{A})A$. Since $KA\subseteq aA+bA$, the equality holds.
\end{proof}

\subsection{Case $\gamma=0$}
\noindent
It is known that if $\gamma=0$, the isomorphism $A(\alpha,\beta,0)\cong \K[u][\omega;\theta][d;\sigma,\delta]$ holds, for certain automorphisms $\theta,\sigma$ and a $\sigma$-derivation $\delta$ depending on $\alpha$, $\beta$, see \cite[Theorem 3.3]{Kirkman}. We will present an alternative proof of the existence of this isomorphism using the universal property of skew polynomial rings that will be more suitable for our purpose.
\begin{lemma}\label{L2}
For the down-up algebra $A=A(\alpha,\beta,0)$, there exist $\omega$ and automorphisms $\theta$ and $\sigma$, together with a $\sigma$-derivation $\delta$, such that $A$ is isomorphic to an iterated skew polynomial ring of the form $\K[u][\omega;\theta][d;\sigma,\delta]$.
\end{lemma}
\begin{proof}
As usual we denote by $U$ and $D$ the obvious generators of $A$. Let $R_1:=\K[u]$ and $R:=\K[u][\omega;\theta]$, where $\theta$ is the automorphism of $R_1$ given by $\theta(u)=\mu^{-1}u$. If $\phi_0:R_1\to A$ is defined by $\phi_{0}(u)=U$, then $\phi_0$ can be extended to a ring homomorphism with the property that, for $a\in R_1$ the following holds:
\begin{align*}
\phi_0(u)[D,U]_{\lambda}=&UDU-\lambda U^2D=\mu^{-1}(\mu UDU-\lambda\mu U^2D)\\
=&\mu^{-1}(DU-\lambda UD)U=[D,U]_{\lambda}\phi_{0}(\theta(u)).
\end{align*}
Taking $y:=[D,U]_{\lambda}$, in \cite[\S 1.2.5]{McConnell} asserts that there exists a unique ring homomorphism $\phi_1:R_{1}[\omega;\theta]\to A$ such that $\phi_1\circ \iota=\phi_0$, with $\iota:R_1\to R_1[\omega;\theta]$ the natural inclusion. Now, consider the ring $R[d;\sigma,\delta]=\K[u][\omega;\theta][d;\sigma,\delta]$, where $\sigma:R\to R$ is the automorphism $\sigma(u)=\lambda^{-1}u$, $\sigma(\omega)=\mu^{-1}\omega$, and $\delta$ the $\sigma$-derivation on $R$ determined by $\delta(u)=-\lambda^{-1}\omega$ and $\delta(\omega)=0$. For the aforementioned homomorphism $\phi_1$, note that
\begin{align*}
\phi_1(u)D=UD=&\lambda^{-1}DU-\lambda^{-1}[D,U]_{\lambda}\\
=&D\phi_{1}(\sigma(u))+\phi_{1}(\delta(u));\\
\phi_1(\omega)D=&DUD-\lambda D^2=\mu^{-1}(D^2U-\lambda DUD)=\mu^{-1}D[D,U]_{\lambda}\\
&=D\phi_1(\sigma(\omega));
\end{align*}
if we set now $y:=D$, again from \cite[\S 1.2.5]{McConnell}, we obtain a unique ring homomorphism $\phi_2:R[d;\sigma,\delta]\to A$ such that $\phi_2\circ \iota'=\phi_1$, with $\iota':R\to R[d;\sigma,\delta]$ the inclusion. In particular, $\phi_2(u)=U$, $\phi_2(\omega)=[D,U]_{\lambda}$ and $\phi_2(d)=D$; thus, $\phi_2$ is surjective. Since $\B=\{u^i\omega^jd^k\mid i,j,k\in\mathbb{N}\}$ is a $\K$-basis of $R[d;\sigma,\delta]$ and $\phi_1(\B)=\{U^i(DU-\lambda UD)^jD^k\mid i,j,k\in\mathbb{N}\}$, we know that $\phi_2$ is an isomorphism and we have proved the statement.
\end{proof}
\begin{proposition}\label{P1}
The down-up algebra $A(\alpha,\beta,0)$ has a non-trivial, stably free right ideal.
\end{proposition}
\begin{proof}
By Lemma \ref{L2}, the isomorphism $A\cong \K[u][\omega;\theta][d;\sigma,\delta]$ holds, then it is enough to show that the latter ring satisfies the statement. Taking $r=1+u\omega$, we have that
\begin{align*}
r(1-u\omega)-\delta(r)(\mu^{-2}\lambda u^2)=&(1+u\omega)(1-u\omega)-(-\lambda^{-1}\mu^{-1}\omega^2)(\mu^{-2}\lambda u^2)\\
=& 1-u\omega u\omega +\mu^{-3}\omega^2u^2\\
=&1-\mu^{-3}\omega^2u^2+\mu^{-3}\omega^2u^2=1;
\end{align*}
since $A$ is a domain and $\sigma$ is an automorphism,  by Lemma \ref{L0} the algebra $A$ has a stably free right ideal $K$ which is non-free. The ideal $K$ is given by $\{f\in A\mid rf\in dA\}$ and turns out to be isomorphic to $rA\cap dA$. We assert that $K$ is generated by the polynomials $a=d^2$, and $b=du\omega+\lambda^{-1}\mu\omega^2+\mu^{2}d$: indeed, we start noting that
\begin{align*}
ra=&(u\omega +1)d^2=uwd^2+d^2\\
=&\mu^{-2}\lambda^{-2}d^2u\omega-(\lambda^{-2}\mu^{-2}+\lambda^{-1}\mu^{-3})d\omega^2\in dA,
\end{align*}
and
\begin{align*}
rb=&(u\omega+1)(du\omega+\lambda^{-1}\mu\omega^2+\mu^2d)\\
=&(d\sigma(u\omega)+\delta(u\omega))u\omega +\lambda^{-1}\mu u\omega^{3}+\mu^2(d\sigma(u\omega)+\delta(u\omega))+du\omega+\lambda^{-1}\mu\omega^2+\mu^2d\\
=&(\lambda^{-1}\mu^{-1}du\omega-\lambda^{-1}\mu^{-1}\omega^2)u\omega+\lambda^{-1}\mu u\omega^3+\mu^2(\lambda^{-1}\mu^{-1}du\omega-\lambda^{-1}\mu^{-1}\omega^2)+ du\omega+\lambda^{-1}\mu\omega^2+\mu^2d\\
=&d(\lambda^{-1}u^2\omega^2+(\lambda^{-1}\mu+1)u\omega+\mu^2)-\lambda^{-1}\mu^{-1}\omega^2u\omega+\lambda^{-1}\mu u\omega^3-\lambda^{-1}\mu\omega^2+\lambda^{-1}\mu\omega^2\\
=&d(\lambda^{-1}u^2\omega^2+(\lambda^{-1}\mu+1)u\omega+\mu^2)\in dA.
\end{align*}
In order to prove the claim, we suppose that $\B$ is ordered by the deglex order $\prec$ with $u\prec \omega\prec d$. Let $f$ be a non-zero polynomial in $A$. We shall show that if $f\in K$, then $lm(f)$ is divisible by $lm(a)=d^2$ or $lm(b)=u\omega d$. Let $lm(f)=u^{\delta_1}\omega^{\delta_2}d^{\delta_3}$ be the leading monomial of $f$. A straightforward reasoning allows to derive that if $f\in K$, then necessarily $\delta_3\geq 1$. We consider the following possibilities:
\begin{itemize}
\item $\delta_1=\delta_2=0$: in such case $\delta_3\geq 2$, since otherwise $f=c_1d+c_2\omega+c_3u+c_4$ with $c_i\in \K$, $i=1,2,3,4$ and $c_4$ non-zero. So,
\begin{align*}
(u\omega+1)f=&d(c_1\lambda^{-1}\mu^{-1}u\omega+c_1)+c_2u\omega^2+c_3\mu u^2\omega-c_1\lambda^{-1}\mu^{-1}\omega^2+c_4u\omega+c_2\omega+c_3u+c_4\in dA
\end{align*}
implies that $c_1=c_2=c_3=c_4=0$; i.e., $f=0$ which is contrary to our choice of $f$. Thus, $\delta_3\geq 2$.
\item $\delta_3=1$: in this situation we must have $\delta_1,\delta_2\geq 1$. By the above, we get that either $\delta_1$ or $\delta_2$ is not zero. Suppose $\delta_1\neq 0$ and $\delta_2=0$. Thus $f=cu^{\delta_1}d+f_1$ with $lm(f_1)\prec lm(f)$ and $c\neq 0$; then
\begin{align*}
(u\omega+1)f=&c\mu^{\delta_1}(d\sigma(u^{\delta_1+1}\omega)+\delta(u^{\delta+1}\omega))+c(d\sigma(u^{\delta_1})+\delta(u^{\delta_1}))+u\omega f_1+f_1\\
=&c\mu^{\delta_1}(\lambda^{-(\delta_1+1)}\mu^{-1}du^{\delta_1+1}\omega-\lambda^{-1}\mu^{-1} p_{\delta_1+1}(\lambda^{-1},\mu)u^{\delta_1}\omega^2)
+c\lambda^{-\delta_1}du^{\delta_1}\\
&-c\lambda^{-1}p_{\delta_1}(\lambda^{-1},\mu)u^{\delta_1-1}\omega+u\omega f_1+f_1\\
=&d(c\lambda^{-(\delta_1+1)}u^{\delta_1+1}\omega+c\lambda^{-\delta_1}u^{\delta_1})-c\lambda^{-1}\mu^{\delta_1-1}p_{\delta_1+1}(\lambda^{-1},\mu)u^{\delta_1} \omega^2\\
&-c\lambda^{-1}p_{\delta_1}(\lambda^{-1},\mu)u^{\delta_1-1}\omega+u\omega f_1+f_1\in dA,
\end{align*}
where $p_t(\lambda^{-1},\mu)=\lambda^{-(t-1)}\mu^{t-1}+\lambda^{-(t-2)}\mu^{t-2}+\cdots+\lambda^{-1}\mu+1$ is the expression that appears in the calculation of $\delta(u^t)$ and $\delta(u^{t}\omega)$. Specifically, using induction over $t\geq 1$, it can be shown that $\delta(u^t)=-\lambda^{-1}p_t(\lambda^{-1},\mu)u^{t-1}\omega$, and  $\delta(u^{t}\omega)=-\lambda^{-1}\mu^{-1}p_t(\lambda^{-1},\mu)u^{t-1}\omega^2$ for $t\geq 1$. Let $\epsilon_1=c\lambda^{-1}\mu^{\delta_1-1}p_{\delta_1+1}(\lambda^{-1},\mu)$, $\epsilon_2=c\lambda^{-1}p_{\delta_1}(\lambda^{-1},\mu)$ and $c'\in\K$ the coefficient of $u^{\delta_1-1}\omega$ in $f_1$. Thus $\mu^{\delta_1-1} c'-\epsilon_1=0$ and $c'-\epsilon_2=0$. Rewriting these equations, we obtain that $cp_{\delta_1+1}(\lambda^{-1},\mu)=cp_{\delta_1}(\lambda^{-1},\mu)$. Hence $c(p_{\delta_1+1}(\lambda^{-1},\mu)-p_{\delta_1}(\lambda^{-1},\mu))=0$; but $p_{\delta_1+1}(\lambda^{-1},\mu)-p_{\delta_1}(\lambda^{-1},\mu)=\lambda^{-\delta_1}\mu^{\delta_1}$ and, since $\lambda$ and $\mu$ are non-zero, it follows that $c=0$. This is a contradiction, therefore $\delta_2\geq 1$.
\item $\delta_3=1$ and $\delta_2\neq 0$. If $\delta_1=0$, the polynomial $f$ is written as $f=c\omega^{\delta_2}d+f_1$ with $lm(f_1)\prec lm(f)$ and $c\neq 0$. In this case
\begin{align*}
(u\omega+1)f=&c(d\sigma(u\omega^{\delta_2+1})+\delta(u\omega^{\delta_2+1}))+cd\sigma(\omega^{\delta_2})+u\omega f_1+f_1\\
=&c\lambda^{-(\delta_2+1)}\mu^{-1}du\omega^{\delta_2+1}-c\lambda^{-1}\mu^{-(\delta_2+1)}\omega^{\delta_2+2}+c\mu^{-\delta_2}d\omega^{\delta_2}+u\omega f_1+f_1\\
=&d(c\lambda^{-(\delta_2+1)}\mu^{-1}u\omega^{\delta_2+1}+c\mu^{-\delta_2}\omega^{\delta_2})-c\lambda^{-1}\mu^{-(\delta_2+1)}\omega^{\delta_2+2}+u\omega f_1+ f_1\in dA.
\end{align*}
Since each term in $u\omega f_1$ is multiplied by $u$ and deg($f_1$) $\leq \delta_2+1$, it is necessary that $c\lambda^{-1}\mu^{-(\delta_2+1)}=0$. Thus $c=0$, which contradicts our choice of $f$. Consequently, $\delta_1\geq 1$ .
\end{itemize}
Therefore, given  a non-zero polynomial $f\in K$ and applying a right division algorithm, there exist $q_1,q_2,h\in A$ such that $f=aq_1+bq_2+h$, with $h$ reduced with respect to $a$ and $b$. If $h\neq 0$, then $h$ is not divisible neither by $lm(a)$ nor by $lm(b)$. But $h=f-aq_1-bq_2\in K$ and we get a contradiction. In consequence $h=0$, $f=aq_1+bq_2$ and $K=aA+bA$.
\end{proof}
\begin{remark}
It is proved in \cite{Jordan} that a down-up algebra is isomorphic to an ambiskew ring. Specifically, in that paper it is showed that $A(\alpha,\beta,\gamma)\cong \K[\omega][u,\sigma][d;\sigma^{-1},\delta]$, where $\omega=du-\lambda ud$, $\sigma$ is the automorphism over $\K[\omega]$ given by $\sigma(\omega)=\mu\omega +\gamma$ extended to $\K[\omega][u;\sigma]$ by setting $\sigma(u)=\lambda u$, and $\delta$ the $\sigma^{-1}$-derivation with $\delta(\K[\omega])=0$ and $\delta(u)=-\lambda^{-1}\omega$. We could have tried to apply directly the results from $\cite{Stafford5}$ to this ring in order to obtain a non-free, stably free right ideal. Nevertheless, despite this isomorphism, the element $r$ in Lemma \ref{L0} cannot be attained in a natural way using this approach; for this reason we decided to consider the cases $\gamma\neq 0$ and $\gamma=0$ independently.
\end{remark}
\section{Stable rank}\label{s4}
\noindent
In this last section we assume additionally that $\K$ is an algebraically closed field. Recall that a \-u\-ni\-mo\-du\-lar row $\vu=(u_1,u_2,\ldots,u_s)$ with entries in a ring $S$ is said to be \emph{stable} if there exist $r_1,\ldots,r_{s-1}$ such that $\vu'=(u_1+u_sr_1,\ldots,u_{s-1}+u_sr_{s-1})$ is also a unimodular row. The \emph{stable rank} of $S$ is defined as the least non-negative integer $t$ with the property that every unimodular row of length $t+1$ is stable, see \cite{Lam6}, \cite[Chapter 11]{McConnell} and references therein for features and interesting examples of stable rank. Furthermore, if sr($S$) denotes the stable rank of $S$, the Stafford's Stable Range Theorem states that if $S$ is right noetherian and rKdim($S$) $=d<\infty$, where rKdim($S$) denotes the right Krull dimension of $S$ in the sense of Rentschler and Gabriel, then sr($S$) $\leq$ rKdim($S$)$+1$, see \cite{Stafford3}.\\
\\
In \cite[Theorem 4.1]{Bavula4} Bavula and Lenagan showed that if $A=A(\alpha,\beta,\gamma)$ is a down-up algebra with $\beta\neq 0$, then the right Krull dimension of $A$ is equal to 2 if and only if char($\K$)$=0$, $\alpha+\beta=1$ and $\gamma\neq 0$; otherwise, the right Krull dimension of $A$ is 3. Since $A\cong A^{op}$ via the map $D\mapsto U^{\circ}$ and $U\mapsto D^{\circ}$ \cite[\S1]{Kirkman}, we have that rKdim$(A)=$ lKdim($A$) thus we will simply refer to Kdim$(A)$. These values of Kdim($A$), combined with results of the previous section, will allow us to establish bounds of sr$(A)$. Before doing this, recall that if $\K$ is a field, $\K_0$ its prime subfield and $t$ the transcendence degree of $\K$ over $\K_0$, then the \emph{Kronecker dimension} of $\K$ is defined to be $t$ if char($\K$) $>0$, and $t+1$ if char($\K$) $=0$.
\begin{theorem}\label{P2}
Let $A=A(\alpha,\beta,\gamma)$ a noetherian down-up algebra. We have the following bounds of sr$(A)$:
\begin{enumerate}[(i)]
\item If $\alpha+\beta=1$ and $\gamma\neq 0$ then  $2\leq$ sr($A$) $\leq 3$.
\item If $\alpha+\beta=1$ and $\gamma=0$ then $3\leq$ sr($A$) $\leq 4$.
\item Otherwise, $2\leq$ sr($A$) $\leq 4$.
\end{enumerate}
\end{theorem}
\begin{proof}
Given an arbitrary ring $S$, it is well known that if $M$ is a stably free $S$-module and rank($M$) $\geq$ sr($S$), then $M$ is free with dimension equal to rank($M$), see \cite[Theorem 11.3.7 (i)]{McConnell}. In consequence, by Corollary \ref{C1} and Proposition \ref{P1} we get that sr$(A)\geq 2$ for any noetherian down-up algebra $A$. Because Kdim$(A)=2$ for the case (i), the Stable Range Theorem asserts that sr($A$) $\leq 3$ and the inequality is obtained.\\
For (ii), we have that $2\leq$ sr($A$) $\leq 4$; however, in \cite[Proposition 4.2]{carvalho} the authors proved that $\K[x,y]\cong A/I$ for some two-sided ideal $I$  of $A$ when $\alpha+\beta=1$ and $\gamma=0$. Since the stable rank of a quotient is not bigger than the stable rank of the ring, it follows that  sr$(\K[x,y])\leq$ sr($A$). Furthermore, Suslin proved in \cite[Theorem 10]{Suslin} that if $l$ is the Kronecker dimension of $\K$ and $n\leq l$, then sr($\K[x_1,\ldots,x_n]$) $=n+1$. In particular, sr($\K[x,y]$) $=3$ and $3\leq$ sr($A$) $\leq 4$.
\end{proof}

Computing the exact value of the stable rank of an arbitrary ring is a very difficult task. However, as a significant example, the stable rank of commutative polynomial rings over fields was determined by Suslin \cite{Suslin}. In the noncommutative setting, it is well known that the stable rank of the $n$-th Weyl algebra $A_n(\K)$ is 2 when char($\K$) $=0$ \cite{Stafford2}. Tintera  showed in \cite{Tintera} that if $\h$ is the Heisenberg Lie algebra of dimension $n$ over a field $\K$ ``large enough'', then sr($\U(\h)$) $=n$. The Kmax dimension of a ring was the main tool used by him in order to obtain this equality. We will use an analogous argument for computing the exact value of sr($A$), when $A$ is a down-up algebra with  $\alpha+\beta=1$ and $\gamma=0$.\\
\\
In order to recall the definition of the Kmax dimension, we begin by considering the deviation of a poset \cite[Section 6.1]{McConnell}: let $\PP$ be a poset, $a,b\in \PP$ and $a\geq b$. The factor of $a$ and $b$ is the subposet of $\PP$ defined as $\PP_{a,b}=\{x\in \PP\mid a\geq x\geq b\}$. To define the deviation of $\PP$, or dev($\PP$) for short, we say that dev($\PP$) $=-\infty$ if $\PP$ is trivial. If $\PP$ is non-trivial but satisfies the d.c.c., then
dev ($\PP$) $=0$. For a general ordinal $\alpha$, we define dev($\PP$) $=\alpha$ provided:
\begin{enumerate}[(i)]
\item dev($\PP$) $\neq \beta<\alpha$;
\item in any descending chain of elements of $\PP$, all but finitely many factors have deviation less than $\alpha$.
\end{enumerate}
Now, we recall the definition of the Kmax dimension of a ring $S$.
\begin{definition}
Let $S$ be an associative ring with identity and denote by $\M_S$ the set of maximal right ideals of $S$.
\begin{enumerate}[(i)]
\item A right ideal $I$ of $S$ is called a \textbf{Jacobson right ideal} if $I=J(I)$, where $J(I):=\bigcap\{M\in \M_{S}\mid I\subseteq M\}$.
\item Let $\J\LL(S)$ be the set of Jacobson right ideals of $S$ partially ordered by inclusion. The \textbf{Kmax dimension} of $S$ is defined to be the
deviation of the poset $(\J\LL(S),\subseteq)$ and we denote it by Kmax($S$).
\end{enumerate}
\end{definition}
\begin{remark}\label{R4}
(i) Let $(\LL(S),\subseteq)$ be the poset of right ideals of $S$. Note that $(\J\LL(S),\subseteq)$ is a subposet of $(\LL(S),\subseteq)$. So, if rKdim$(S)$ exists, we have that Kmax($S$) $\leq$ rKdim($S$). In particular if $S$ is a right noetherian ring, this inequality always holds.
\\
(ii) There exist rings for which the inequality in (i) is strict: if $S=\mathbb{C} [x]_{(x)}[y]$, with $\mathbb{C}[x]_{(x)}$ denoting the localization of $\mathbb{C}[x]$ at powers of $x$, then Kmax($S$) $=1<$ Kdim($S$) $=2$ (see \cite{Stafford6}, remark to Proposition 1.6). Furthermore, for $\h$ a non-abelian nilpotent Lie algebra of dimension $n$, Tintera proved in \cite[Lemma 3]{Tintera} that Kmax($\U(\h)$) $<$ Kdim($\U(\h)$) $=n$.
\\
(iii) It follows from \cite{Stafford3} or \cite[Theorem B]{Stafford6} that if $S$ is a right noetherian ring, a Kmax version of the Stable Range Theorem holds; i.e., if Kmax($S$) exists and is finite, then sr($S$) $\leq $ Kmax($S$) $+1$. The latter explains the reason leading us to introduce the Kmax dimension.
\end{remark}
Below we summarize some important properties satisfied by the Kmax dimension.

\begin{lemma}\cite[Lemmas 1 and 2]{Tintera}\label{L3}  (i)\, Given a ring extension $S \subset T$ such that $T_S$ is a faithfully flat module, we have Kmax($S$) $\leq $ Kmax($T$).

(ii)\, Let $S$ be a domain for which Kmax($S$) is defined and let $z\in S$ be a normal element. There is an inequality
\begin{center}
Kmax($S$) $\leq$ sup$\{$ Kmax($S/zS$), Kmax($S_{(z)}$)$\}$,
\end{center}
where $S_{(z)}$ denotes the localization of $S$ at the set of powers of $z$.
\end{lemma}
For another features of Kmax, as well as for additional descriptions and remarks, we refer to \cite{Stafford6}.\\
\\
Recall that given a $\K$-algebra $R$, an automorphism $\sigma$ of $R$ and a central element $a$ of $R$, the \-ge\-ne\-ra\-li\-zed Weyl algebra $R(\sigma,a)$ is defined as the algebra generated by $X^+$ and $X^-$ subject to the relations: $X^-X^+=a$, $X^+X^-=\sigma(a)$, $X^-\sigma(b)=bX^-$ and $X^+b=\sigma(b)X^+$ for all $b\in R$. For these algebras, Bavula and van Oystaeyen established in \cite[Theorem 1.2]{Bavula3} the following result for computing the Krull dimension of $T=R(\sigma,a)$ when $R$ is commutative.
\begin{proposition}\label{T2}
Let $R$ be a commutative noetherian ring with Kdim($R$) $=m$ and let $T=R(\sigma,a)$ be a \-ge\-ne\-ra\-li\-zed Weyl algebra. The Krull dimension Kdim($T$) is $m$ unless there is a height $m$ maximal ideal $P$ of $R$ such that one of the following conditions holds:
\begin{enumerate}[(i)]
\item $\sigma^n(P)=P$, for some $n>0$;
\item $a\in \sigma^n(P)$ for infinitely many $n$.
\end{enumerate}
If there is an ideal $P$ as above such that (i) or (ii) holds, then Kdim($T$) $= m + 1$.
\end{proposition}
To prove the main result of this section, we develop reasonings inspired into those carried out by Carvalho and Musson in \cite[\S 5.]{carvalho2}.\\
\\
In \cite[\S 2.2]{Kirkman} it is showed that an arbitrary noetherian down-up algebra is isomorphic to a \-ge\-ne\-ra\-li\-zed Weyl algebra: in fact, taking $R=\K[x,y]$, $\phi$ the automorphism of $R$ defined by $\phi(x)=y$, $\phi(y)=\alpha y+\beta x+\gamma$ and $a=x$, the algebra $A(\alpha,\beta,\gamma)$ is isomorphic to $R(\phi,x)$ under the isomorphism $\varphi$ sending $X^+$ to $D$ and $X^-$ to $U$; in particular, $x$ and $y$ correspond to $UD$ and $DU$, respectively. Additionally, if  $\alpha+\beta=1$, $\gamma=0$ and the roots $\lambda$ and $\mu$ of $t^2-\alpha t-\beta$ are different, then case 2 of \cite[\S 1.4]{carvalho} holds and we have that
\begin{center}
$\omega_1=\beta x+y$\\
$\omega_2=-x+y$
\end{center}
are such that $\{1,\omega_1,\omega_2\}$ is a basis of the subspace of $\K[x,y]$ generated by $1$, $x$, $y$, and moreover $\phi(\omega_1)=\omega_1$ and $\phi(\omega_2)=-\beta \omega_2$. In this case $\omega_2$ is identified with $\omega=DU-UD$ through $\varphi$.\\
\\
Benkart and Roby introduced in \cite{Benkart1} (see also \cite{carvalho}) the following recursive relation in order to study Verma modules of $A(\alpha,\beta,\gamma)$:
\begin{equation}\label{RR}
s_{n}=\alpha s_{n-1}+\beta s_{n-2}+\gamma.
\end{equation}
From \cite[Lemma 2.3]{carvalho} it follows that for all $n\in \mathbb{Z}$, the automorphism $\phi$ satisfies
\begin{center}
$\phi^{-n}\langle x-s_0,y-s_1\rangle = \langle x-s_n,y-s_{n+1}\rangle$,
\end{center}
where $\langle x-s_0,y-s_1\rangle$ denotes the two-sided ideal generated by $x-s_0$ and $y-s_1$.
For $\alpha^{2}+4\beta\neq 0$ (i.e., when the roots of polynomial $t^2-\alpha t-\beta$ are different) and $\alpha +\beta =1$, the solution to (\ref{RR}) is given by \cite[Proposition 2.12(i)]{Benkart1}:
\begin{equation}\label{RRS}
s_n=c_1\lambda^{n}+c_2\mu^{n}+\frac{\gamma n}{(2-\alpha)}, \text{ (necessarily $\alpha\neq 2$)}
\end{equation}
for certain fixed scalars $c_1,c_2\in \K$ which depend on the established initial conditions.
\begin{lemma}\label{L4}
Let $A(\alpha,\beta,0)$ be a down-up algebra with $\alpha+\beta=1$, such that roots $1,\mu$ are different and $\mu$ is not a root of unity. If $Q$  is a maximal ideal of $\K[x,y]$ such that $x\in \phi^{n}(Q)$ for infinitely many $n$, then $\phi^{n}(Q)=Q$ for some $n\geq 1$.
\end{lemma}
\begin{proof}
Let $Q=\langle x-s_0, y-s_1\rangle$ for certain $s_0,s_1\in \K$. By hypothesis, we can suppose that $x\in Q$, namely $Q=\langle x, y-s_1\rangle$ and $s_0=0$. Using the initial conditions $s_0=0$ and $s_1=s$, the equation (\ref{RRS}) can be written as
\begin{equation*}
s_n= \frac{s}{1-\mu}-\frac{s}{1-\mu}\mu^{n}.
\end{equation*}
If $s_n=0$ for $n\geq 1$, we get $\frac{s}{1-\mu}(1-\mu^n)=0$. Since $\mu$ is not a root of unity, necessarily $s=0$ and $s_n=0$ for all $n$. Thus, the proof is obtained.
\end{proof}

\begin{theorem}\label{P3}
Let $A=A(\alpha,\beta,0)$ be a down-up algebra with $\alpha+\beta=1$. If one the following conditions holds
\begin{enumerate}[(i)]
\item $\lambda=\mu=1$, or
\item $\mu\neq 1$ and it is not a root of unity,
\end{enumerate}
then sr($A$) $=3$.
\end{theorem}
\begin{proof}
In the first case $\alpha=2$, $\beta=-1$ and $A(\alpha,\beta,0)\cong U(\h)$, where $\h$ denotes the Heisenberg algebra of dimension 3. The assertion follows from \cite[Corollary 1]{Tintera}. Suppose $\lambda=1$, $\mu\neq 1$ and $\mu$ is not a root of unity. Under these conditions $\omega=DU-UD$ is a normal element of $A$ and $A/\omega A$ is a ring. Since $A$ is a domain and $\omega$ is not zero, then $\omega$ is a regular element of $A$ and, therefore, Kdim($A/\omega A$) $<$ Kdim($A$) $=3$, see \cite[Lemma 6.3.9]{McConnell}. On the other hand, if $\Lambda=\{\omega^{i}\mid i\in \mathbb{N}\}$, then $\Lambda$ is an Ore set and the ring $A_{(\omega)}:=A\Lambda^{-1}$ exists. Under the isomorphism $A\cong \K[x,y](\phi,x)$ described above, the element $\omega$ is sent to $\omega_2=y-x\neq 0$. Thus, $A_{(\omega)}\cong \K[x,y]_{(\omega_2)}(\phi, x)$ is a generalized Weyl algebra. To prove that Kdim($A_{(\omega)}$) $=$ Kdim($\K[x,y]_{(\omega_2)}$) $=2$, we must show that neither (i) nor  (ii) in Theorem \ref{T2} is satisfied. By Lemma \ref{L4} it is enough to demonstrate that for any maximal ideal $P$ of $\K[x,y]_{(\omega_2)}$ and $n>0$, we have $\phi^{n}(P)\neq P$. This is equivalent to prove that if $Q$ is a maximal ideal of $\K[x,y]$ such that $\sigma^{n}(Q)=Q$, then $\omega_2\in Q$. Thus, for $Q=\langle \omega_1-a_1, \omega_2-a_2 \rangle$ and, since $\mu$ is not a root of unity, from
\cite[Lemma 2.2(ii)]{carvalho} it follows that $a_2=0$ and we get the statement about Kdim($A_{(\omega)}$). Hence, by Lemma \ref{L3}(ii), we have  Kmax($A$) $\leq 2$ and the Stable Range Theorem asserts that sr($A$) $\leq 3$. The equality  follows from Proposition \ref{P2} (ii).
\end{proof}
\begin{remark}\label{R5}
In \cite[Section 5]{carvalho2} it was proved that if $A(\alpha,\beta,\gamma)$ is a down-up algebra such that the polynomial $t^2-\alpha t-\beta$ equals $(t-\mu)^2$, with $\mu\neq 1$ and $\mu$ a primitive $n$-th root of unity, then $\omega=DU-\mu UD+\frac{\gamma}{\mu-1}$ is a normal element of $A$ and Kdim($A_{(\omega)}$) $=2$. From this and Lemma \ref{L3}(ii) we obtain the following.
\end{remark}
\begin{corollary}
Let $A=A(\alpha,\beta,\gamma)$ a down-up algebra such that $\alpha+\beta\neq 1$. If
\begin{enumerate}[(i)]
\item $t^2-\alpha t-\beta=(t-\mu)^2$ with $\mu$ a primitive $n$-th root of unity, or
\item $\gamma=0$, $t^2-\alpha t-\beta=(t-\mu)(t-\lambda)$ with $\lambda\neq \mu$ and $\frac{\lambda}{\mu}$ not a root of unity,
\end{enumerate}
then $2\leq$ sr($A$) $\leq 3$.
\end{corollary}
\begin{proof}
Part (i) follows from Remark \ref{R5}. For (ii), suppose that $\lambda$ is not a root of unity and set $\omega=\frac{1}{\lambda^2-\lambda}(\beta(\lambda-1)UD+\lambda(\lambda-1)DU)$. Therefore, $\omega=DU-\mu UD$ and a straightforward calculation shows that $\omega$  is a normal element of $A$. Thus $A_{(\omega)}$ exists and $A_{(\omega)}\cong \K[x,y]_{(\omega')}(\phi,x)$, where $\omega'=y-\mu x$. Given a maximal ideal $Q=\langle x-s_0,y-s_1 \rangle$ of $\K[x,y]$ such that $x\in \phi^{n}(Q)$ for infinitely many values of $n$, we have that $\phi^{n}(Q)=Q$ for some $n\geq 1$: in fact, using the initial conditions $s_0=0$ and $s_1=s$, in this case the equation (\ref{RRS}) can be written as $s_n= \frac{s}{\lambda-\mu}(\lambda^n-\mu^n)$; if $s_n=0$, necessarily $s=0$ since $\frac{\lambda}{\mu}$ is not a root of unity, and the assertion follows.  A similar reasoning to the one in the proof of Proposition \ref{P3}, along with \cite[Lemma 2.2(i)]{carvalho}, allow to obtain Kdim($A_{(\omega)}$) $=2$ and from Lemma \ref{L3}(ii) we get the inequality sr($A$) $\leq 3$.
\end{proof}
\begin{remark}
The stable rank is closely related to the cancellation property for projective modules. Recall that two finitely generated projective $S$-modules $P$ and $P'$ are called \emph{stably isomorphic} if $P\oplus S^n\cong P'\oplus S^n$ for some $n$. It is said that $P$ satisfies the \emph{cancellation property} if any $P'$ stably isomorphic to $P$ is in fact isomorphic to $P$. Hence, if $P$ is a finitely generated projective module over a down-up algebra $A$ with rank($P$) $\geq$ sr($A$), a simple reasoning proves that $P$ has the cancellation property.
\end{remark}
A table summarizing the stable ranks of some important examples of down-up algebras, and one related algebra, is included below.
\begin{table}
\begin{center}
\begin{tabular}{|p{4.5cm}|p{3.5cm}|p{3cm}|c|}
\hline
\textbf{Algebra} $A=A(\alpha,\beta,\gamma)$ & \textbf{Parameters} & \textbf{Conditions on parameters} & \textbf{Bounds of sr}($A$) \\
\hline \hline
$\U(sl_2(\K))$ & $\alpha=2$, $\beta=-1$, $\gamma=-2$ & & $2\leq$ sr($A$) $\leq 3$  \\ \hline
$\U(osp(1,2))$ & $\alpha=0$, $\beta=1$, $\gamma=\frac{1}{2}$ & & $2\leq$ sr($A$) $\leq 3$  \\ \hline
$\U(\h)$, $\h$ the Heisenberg algebra of dimension 3 & $\alpha=2$, $\beta=-1$, $\gamma=0$ & &  sr($A$) $=3$  \\ \hline
\multirow{2}{4.5cm}{Smith's algebras similar to $U(sl_2(\K))$ with deg($f(h)$) $\leq 1$} & $\alpha=2$, $\beta=-1$ &  $\gamma=0$ & sr($A$) $=3$\\ \cline{3-4}
& & $\gamma\neq 0$& $2\leq$ sr($A$) $\leq 3$ \\ \hline
\multirow{2}{4.5cm}{Conformal $sl_2(\K)$ algebras with $c\neq 0$, $b=0$ \linebreak
$xz-azx=x$, $zy-ayz=y$, $yx-cxy=bz^2+z$} & $\alpha=c^{-1}(1+ac)$, $\beta=-ac^{-1}$, $\gamma=-c^{-1}$ &  $c=1$ or $a=1$ & $2\leq$ sr($A$) $\leq 3$\\ \cline{3-3}
 & & $a=c^{-1}\neq 1$ and $a$  a primitive root of unity & \\ \cline{3-4}
& & Otherwise & $2\leq$ sr($A$) $\leq 4$ \\ \hline
\multirow{3}{4.5cm}{Quantum Heisenberg algebra $H_q\cong \U^{+}_q(sl_3(\K))$, $q\in K^{*}$.\linebreak
$zx=qxz$, $zy=q^{-1}yz$, $xy-qyx=z$} & $\alpha=q+q^{-1}$, $\beta=-1$, $\gamma=0$ and $q\neq 1$ &  $q=-1$ & $2\leq$ sr($A$) $\leq 3$\\ \cline{3-3}
& & $q^2$ not a root of unity &  \\ \cline{3-4}
& & Otherwise & $2\leq$ sr($A$) $\leq 4$ \\ \hline
\multirow{2}{4.5cm}{$q$-analog $H_q'$ of $U(\h)$, $q\neq 0,1$.\linebreak $xy-qyx=z$, $xz=qzx$ and $zy=qyz$} & $\alpha=2q$, $\beta=-q^2$, $\gamma=0$ &  $q$ is a primitive root of unity & $2\leq$ sr($A$) $\leq 3$\\ \cline{3-4}
& & Otherwise & $2\leq$ sr($A$) $\leq 4$ \\ \hline
\multirow{2}{4.5cm}{$\mathfrak{m}(\underline{\xi})$ the Witten's Deformation of $\U(sl_2(\K))$: $xz-\xi_1zx=\xi_2x$, $zy-\xi_3yz=\xi_4y$, $yx-\xi_5xy=\xi_6 z^2+\xi_7 z$, with $\xi_6=0$, $\xi_5\xi_7\neq 0$, $\xi_1=\xi_3$ and $\xi_2=\xi_4$} & $\alpha=\frac{1+\xi_1\xi_5}{\xi_5}$, $\beta=-\frac{\xi_1}{\xi_5}$, $\gamma=-\frac{\xi_2\xi_7}{\xi_5}$ & $\xi_1=\xi_5=1$ and $\xi_2=0$  & sr($A$) $=3$  \\ \cline{3-3}
 & & $\xi_2=0$, $\xi_5=1$, $\xi_1$  not a root of unity & \\ \cline{3-3}
 & & $\xi_1=1$, $\xi_2=0$, $\xi_5$ not a root of unity &  \\ \cline{3-4}
 & & $\xi_2\neq 0$, $\xi_1=1$ or $\xi_5=1$ & $2\leq$ sr($A$) $\leq 3$\\ \cline{3-3}
 & & $\xi_1\neq 1$, $\xi_1=\xi_5^{-1}$ and $\xi_1$ is a primitive root of unity & \\ \cline{3-3}
 & & $\xi_2=0$, $\xi_1,\xi_5\neq 1$ and $\xi_1\xi_5$ not a root of unity & \\ \cline{3-4}
 & & Otherwise & $2\leq$ sr($A$) $\leq 4$ \\ \hline
\multirow{3}{4.5cm}{Woronowicz's algebras: $xz-\zeta^4zx=(1+\zeta^2)x$, $zy-\zeta^4yz=(1+\zeta^2)y$, $xy-\zeta^2yx=\zeta z$ with $\zeta\neq \pm1,0$} & $\alpha=\zeta^2(1+\zeta^2)$, $\beta=-\zeta^6$, $\gamma=\zeta(1+\zeta^2)$ &   & $2\leq$ sr($A$) $\leq 4$ \\
 &&&\\
 &&& \\
 &&& \\ \hline
Universal enveloping of Lie super algebra $sl(1,1)$ & $\U(sl(1,1))\cong \frac{A(0,1,0)}{\langle d^2,u^2\rangle}$ & & sr($\U(sl(1,1)$) $=1$  \\ \hline
\end{tabular}
\caption{Bounds of the stable rank of $A(\alpha,\beta,\gamma)$.}
\label{bounds}
\end{center}
\end{table}

\newpage

\bigskip
\noindent
C.G.:\\
IMAS, UBA-CONICET, Consejo Nacional de Investigaciones Cientícas y Técnicas,\\
Ciudad Universitaria, Pabellón I, 1428 Buenos Aires, Argentina\\
cgallego@dm.uba.ar\\
\\
A.S.:\\
Departamento de Matemática, Facultad de Ciencias Exactas y Naturales, Universidad de Buenos Aires,\\
Ciudad Universitaria, Pabellón I, 1428 Buenos Aires, Argentina; and\\
IMAS, UBA-CONICET, Consejo Nacional de Investigaciones Científicas y Técnicas, Argentina\\
asolotar@dm.uba.ar


\begin{thebibliography}{200}


\bibitem{Antoniou}Antoniou, I., Iyudu. N. and Wisbauer, R. \textit{On Serre's problem for RIT algebras}, Comm. Algebra. \textbf{31} (12), (2003), 6037-6050.

\bibitem{Bavula3}Bavula, V. and van Oystaeyen, F. \textit{Krull Dimension of Generalized Weyl Algebras and Itered Skew Polynomial Rings: Commutative Coefficients}, J. Algebra \textbf{208}, (1998), 1-34.

\bibitem{Bavula4}Bavula, V. and Lenagan, T. \textit{Generalized Weyl Algebras Are Tensor Krull Minimal}, J. Algebra \textbf{239}, (2001), 93-111.

\bibitem{Benkart1}Benkart, G. and Roby, T.  \textit{Down-Up Algebras}, J. Algebra \textbf{209}, (1998), 305-344.

\bibitem{carvalho}Carvalho, P. and Musson, I., \textit{Down-Up Algebras and their Representation Theory}, J. Algebra \textbf{228}, (2000), 286-310.

\bibitem{carvalho2}Carvalho, P. and Musson, I., \textit{Monolithic modules over noetherian rings}, Glasgow Math. J. \textbf{53}, (2011), 683-692.

\bibitem{solotar}Chouhy, S., Herscovich, E. and Solotar, A., \textit{Hochschild homology and cohomology of
down-up algebras}, arXiv:1609.09809v1 [math.KT].

\bibitem{Iyudu}Iyudu, N. and Wisbauer, R., \textit{Non-trivial stably free modules over crossed products}, J. Phys. A: Math. Theor. \textbf{42}, (2009), 1-11.

\bibitem{Jordan}Jordan, D. A., \textit{Down-Up Algebras and Ambiskew Polynomial Rings}, J. Algebra \textbf{228}, (2000), 311-346.

\bibitem{Kirkman}Kirkman, R., Musson, I. and Passman, D., \textit{Noetherian Down-Up Algebras}, Proc. Am. Math. Soc. \textbf{127}(11), (1999), 3161-3167.

\bibitem{Lam}Lam, T.Y., \textit{Serre's Problem on Projective Modules}, Springer Monographs in Mathematics, Springer, 2006.

\bibitem{Lam6}Lam, T.Y., \textit{A crash course on stable range, cancellation, substitution, and exchange}, J. Algebra Appl. \textbf{3}(03), (2004), 301-343.

\bibitem{McConnell}McConnell, J. and Robson, J., \textit{Noncommutative Noetherian Rings}, Graduate Studies in Mathematics, 30, AMS, 2001.

\bibitem{Stafford2}Stafford, J.T., \textit{Module structure of Weyl algebras}, J. London Math. Soc. \textbf{18}, (1978), 429-442.

\bibitem{Stafford3}Stafford, J.T., \textit{On the stable range of right noetherian rings}, Bull. London Math. Soc. \textbf{13}, (1981), 39-41.

\bibitem{Stafford5}Stafford, J.T., \textit{Stable free, projective right ideals}, Compositio Math. \textbf{54}, (1985), 63-78.

\bibitem{Stafford6}Stafford, J.T., \textit{Absolute stable rank and quadratic forms over noncommutative rings}, $K$-Theory \textbf{4}, (1990), 121-130.

\bibitem{Suslin}Suslin, A. A., \textit{The cancellation problem for projective modules and related topics}, Ring theory (Proc. Conf., Univ. Waterloo, Waterloo, 1978), pp. 323-338, Lecture Notes in Math., Vol. 734, Springer-Verlag, Berlin-New York, 1979.

\bibitem{Tintera}Tintera, G., \textit{Kmax and stable rank of enveloping algebras}, Proc. Amer. Math. Soc. \textbf{119}(3), (1993), 691-696.

\bibitem{zhao}Zhao, K. \textit{Centers of Down-Up Algebras}, J. Algebra  \textbf{214}, (1999), 103-121.

\end{thebibliography}
\end{document}